\documentclass{article}
\usepackage{gensymb}

\usepackage{amsthm}
\usepackage{amssymb}
\usepackage{amsmath}
\usepackage{mathabx}
\usepackage{biblatex}
\usepackage{quiver}
\newtheorem{proposition}{Proposition}[section]
\newtheorem{theorem}{Theorem}[section]
\newtheorem{corollary}{Corollary}[theorem]
\newtheorem{lemma}[theorem]{Lemma}

\theoremstyle{definition}
\newtheorem{definition}{Definition}[section]
\newtheorem{remark}{Remark}[section]

\newtheorem{example}{Example}[section]

\newcommand{\func}[3]{#1 \colon #2 \rightarrow #3}

\newcommand{\acts}[0]{\text{ }\rotatebox[origin=c]{180}{$\circlearrowright$} \text{ }}
\usepackage{graphicx}
\title{Residual Finiteness of Graph Wreath Products}
\author{Amy Needham}
\date{September 2025}

\begin{document}

\maketitle
\begin{abstract}
    We prove necessary and sufficient conditions for when graph wreath products are residually finite, generalising known results for the permutational wreath product and free product cases
\end{abstract}
\section{Introduction}
The graph wreath product of groups is a natural, common generalisation of both the wreath product and the free product. In \cite{KM16}, Kropholler and Martino found sufficient (and some necessary) conditions for certain finiteness properties of graph wreath products. In this paper we study a further finiteness property for graph wreath products, namely residual finiteness, giving necessary and sufficient conditions under which a graph wreath product of two groups, $\Gamma$ and $\Delta$, is residually finite in terms of properties of the graph and the two groups.
\\The behaviour of residual finiteness under the free product and the wreath product is well-understood. It is well known that the free product of two residually finite groups is always residually finite, however in the case of wreath products, there are restrictive criteria due to Gr\"unberg \cite{G57} in the regular case, and Cornulier \cite{C14} in the general case, as to when the resulting group is residually finite.\\ The theorem we shall prove restricts to give the aforementioned results in the case of the empty graph on $\Gamma$ (for free products), and the complete graph on a set (for wreath products).
\begin{definition}
    A group, $G$ is \textit{residually finite} if for every $g \in G \setminus \{e\}$, there is a finite group $Q$ and a homomorphism $\func{\phi}{G}{Q}$ such that $\phi(g) \neq e_Q$. 
\end{definition}
Specifically, we will prove the following: \begin{theorem}
    For groups $\Gamma, \Delta$ a graph $G = (V, E)$ on which $\Gamma$ acts, $G(\Delta) \rtimes \Gamma$ is residually finite if and only if \begin{enumerate}
        \item $\Gamma, \Delta$ are residually finite;
        \item Either $\Delta$ is abelian and for all neighbouring $v, w \in V$ there is a finite index subgroup $K \leq \Gamma$ such that $w \notin Kv$, or for all $v$ there is some finite index subgroup $K \leq \Gamma$ such that $Kv \cap N(v) = \emptyset$;
        \item for all $v, w \in X$ not neighbouring and not equal there is a finite index $K \leq \Gamma$ such that $Kw \cap (N(v) \cup \{v\}) = \emptyset$.
        
    \end{enumerate}
\end{theorem}
We will recall the definition of $G(\Delta) \rtimes \Gamma$ in the preliminaries.
 As we will prove in the preliminaries $\Gamma, \Delta \leq G(\Delta) \rtimes \Gamma$, so both $\Gamma, \Delta$ must be residually finite for $G(\Delta) \rtimes \Gamma$ to be residually finite. This generalises Gr\"unberg's theorem on the residual finiteness of regular wreath products \cite{G57}. \begin{theorem}
    A regular wreath product, $\Delta \wr \Gamma$ is residually finite if and only if both $\Delta$ and $\Gamma$ are residually finite, and $\Delta$ is abelian or $\Gamma$ is finite. 
\end{theorem} The techniques needed to prove Theorem 1.1 are largely generalisations of those used to prove Theorem 1.2. Theorem 1.1 also generalises Proposition 3.2 in \cite{C14} which states: \begin{theorem}
    A wreath product $\Delta \wr_X \Gamma$ is residually finite if and only if $\Delta, \Gamma$ are residually finite and either: \begin{enumerate}
        \item $\Delta$ is abelian and all vertex stabilisers are profinitely closed in $\Gamma$, or
        \item all vertex stabilisers have finite index in $\Gamma$.
    \end{enumerate}
\end{theorem}
As a further special case of Theorem 1.1, we may deduce the known result that the free product of residually finite groups is residually finite. If we take $G$ to be the graph with vertices $\Gamma$ and no edges, then $G(\Delta) \rtimes \Gamma \cong \Delta *\Gamma$. As $N(\gamma) = \emptyset$ for all vertices $\gamma$ of $G$, condition 1.1(2) is trivial and condition 1.1(3) is equivalent to $\Gamma$ being residually finite, so the result follows.\\
To conclude the introduction, we give three examples (one of a residually finite group, and two of non-residually finite groups) which follow easily from Theorem 1.1, but do not necessarily follow quickly from the residual finiteness of graph products \cite{G90}, or Theorems 1.2, 1.3.
\begin{example}
    Let the graph $G$ have vertices corresponding to $\mathbb Z$ and edges of the form $(i, i+1)$ for all $i \in \mathbb Z$ \[\begin{tikzcd}[ampersand replacement=\&]
	\dots \& -1 \& 0 \& 1 \& \dots
	\arrow[no head, from=1-1, to=1-2]
	\arrow[no head, from=1-2, to=1-3]
	\arrow[no head, from=1-3, to=1-4]
	\arrow[no head, from=1-4, to=1-5]
\end{tikzcd}\] and let $\Gamma = \mathbb Z$ act on $G$ by translation, and let $\Delta$ be any non-trivial residually finite group. Then condition 1.1(1) is met as $\mathbb Z$ is residually finite, condition 1.1(2) is met with $K = 2\mathbb Z$ (in both the case where $\Delta$ is abelian and where it is not) and given vertices $n, m \in G$, 1.1(3) is met by taking $K = k\mathbb Z$ for $k > |n - m| + 1$. Therefore $G(\Delta) \rtimes \mathbb Z$ is residually finite.
\end{example}
\begin{example}
    Consider now the case of $\Delta = S_3$, $\Gamma = \mathbb Z$ and $G$ the graph with vertices corresponding to $\mathbb Z$, and edges of the form $(i, i + n!)$ for $i\in \mathbb Z, n \in \mathbb N$. We claim $G(S_3)\rtimes \mathbb Z$ is not residually finite. Consider condition 1.1(2) and take $v = 0$. As $\Delta$ is non-abelian, if $G(S_3) \rtimes \mathbb Z$ was residually finite, then there would be some finite index subgroup $K \leq \mathbb Z$ containing no element of the form $n!$ for $n \in \mathbb N$. It is well known that $K$ must be of the form $m\mathbb Z$ for some $m \in \mathbb Z$, but then $m! \in K$, contradiction, so $G(S_3) \rtimes \mathbb Z$ is not residually finite.
\end{example}
\begin{example}
    Finally, consider the case of $\Delta$ non-trivial but residually finite, $\Gamma = \mathbb Z$, and $G$ the graph on $\mathbb Z$ with edges $(i, i+1 + n!)$ for $i \in \mathbb Z$ and $n \in \mathbb N$. Then $4\mathbb Z \cap N(0) = \emptyset$ so 1.1(2) is satisfied, but $1 \notin N(0)$, whereas by the same argument as Example 1.2, $K + 1 \cap N(0) \neq \emptyset$ for all finite index $K \leq \mathbb Z$, so by condition 1.1(3), $G(\Delta) \rtimes \Gamma$ is not residually finite.
\end{example}

\section{Preliminaries}
All graphs in this paper will be simplicial graphs (ie all edges are between two distinct vertices, and an edge is uniquely determined by its vertices). For a given vertex $v$ in the vertex set, we denote its set of neighbours by $N(v)$.\\
Throughout the remainder of the paper, $\Gamma$ and $\Delta$ will be groups and $G$ will be a graph on which $\Gamma$ acts (on the left) via graph automorphisms with vertex set $V$ and edge set $E$.\\ We denote by $[g, h]$ the commutator of $g$ and $h$ which we will take to be $ghg^{-1}h^{-1}$.
\begin{definition}
    Given an action of $\Gamma \acts G = (V, E)$ as above and a subgroup $K \leq \Gamma$, we define the \textit{quotient graph} of $G$ by $K$, denoted $K\setminus G$, to be the graph with vertices the orbits of $V$ under $K$ such that there is an edge between $Kv, Kw \in K \setminus G$ if there are some $k, k' \in K$ such that $(kv, k'w) \in E$.
\end{definition}
\begin{definition}
    The \textit{graph product} of $\Delta$ over $G$, denoted $G(\Delta)$, is defined as $$G(\Delta) = \frac{\left(\displaystyle\mathop\Asterisk_{v \in V} \Delta_{(v)}\right)}{\langle\langle [\Delta_{(v)}, \Delta_{(w)}] \mid (v, w) \in E\rangle\rangle}$$ where the $\Delta_{(v)}$ for $v \in V$ are $V$-many isomorphic copies of $\Delta$. We call the relations of the form $[\Delta_{(v)}, \Delta_{(w)}]$ the \textit{commutativity relations}. For a given element $w \in G(\Delta)$, we define a \textit{support}, denoted $\operatorname{supp} (w)$ of $w$ to be a finite subset $W \subseteq V$ such that $w$ lies in the subgroup of $G(\Delta)$ generated by all the $\Delta_{(v)}$ for $v \in W$. In the sequel we denote this $\operatorname{supp}(w)$.
\end{definition}
\begin{remark}
    It is proven in \cite{G90} that there is a unique minimal support for any $w \in G(\Delta)$, although we will not need that fact in this paper. For our purposes, it only matters that supports exist.
\end{remark}
Now as $\Gamma \acts G$, we may define an action $\Gamma \acts G(\Delta)$ by $\gamma \cdot g_{(v)} = g_{\gamma \cdot v}$.
\begin{lemma}
    The above action is a well defined action of $\Gamma$ on $G(\Delta)$ by group automorphisms.
\end{lemma}
\begin{proof}
    The action of $\Gamma$ on the $V$-fold free product of $\Delta$ given by $\gamma \cdot \Delta_{(v)} = \Delta_{(\gamma \cdot v)}$ is well defined by the universal property of the free product, so we get a group automorphism of $\Gamma$ into  maps from the $V$-fold free product to $G(\Delta)$, and this descends to a map into $\operatorname{Aut}(G(\Delta))$ as the fact that $\Gamma$ acts by graph automorphisms ensures that $\gamma \cdot [\Delta_{(v)}, \Delta_{(w)}] = e$ for $(v, w) \in E$.
\end{proof}
We may therefore define a semidirect product.
\begin{definition}
    The \textit{graph wreath product} of $\Delta$ and $\Gamma$ over $G$ is the semidirect product $G(\Delta) \rtimes \Gamma$ given by the action of Lemma 2.1. 
\end{definition}
\begin{remark}
    We briefly indicate how to prove our claim from the introduction that free products are graph wreath products on the empty graph. Suppose we have groups $\Delta$ and $\Gamma$, and we let $G$ be the graph on $\Gamma$ with no edges. Let $\Gamma \acts G$ via left translation. Then $G(\Delta) \rtimes \Gamma$ is the free product of $\Gamma$-many copies of $\Delta$, semidirect product with $\Gamma$. It can be shown that the map sending the $\Delta_{(\gamma)}$ component to $\gamma \Delta \gamma^{-1}$ in $\Gamma * \Delta$ (and sending $\Gamma$ to itself via the identity) is an isomorphism (recalling that the set of all $\gamma \Delta \gamma^{-1}$ with $\gamma \in \Gamma$ generates their free product), and this proves the claim.
\end{remark}
From the definition, it is clear that $\Gamma \leq G(\Delta) \rtimes \Gamma$, and so if $G(\Delta) \rtimes \Gamma$ is residually finite, then $\Gamma$ is. We also similarly get that $G(\Delta)$ is residually finite, but it is not immediate that this implies $\Delta$ is also residually finite. We shall show that, given an induced subgraph $H \subseteq G$, the natural homomorphism $\iota \colon H(\Delta) \rightarrow G(\Delta)$ induced by the identity mappings sending the $\Delta_{(h)}$ component in $H(\Delta)$ to the $\Delta_{(h)}$ component in $G(\Delta)$ is injective. The following lemma is needed to prove so.
\begin{lemma}
    Suppose $H \leq G$ is an induced subgraph. Then $H(\Delta)$ is isomorphic to the subgroup of $G(\Delta)$ generated by the $\Delta$-copies corresponding to the vertices of $H$ via the homomorphism $\iota$ defined above.
\end{lemma}
\begin{proof}
    We have a homomorphism $\iota \colon H(\Delta) \rightarrow G(\Delta)$ via $\Delta_{(v)} \rightarrow \Delta_{(v)}$, as the commutativity relations between $H$-copies of $\Delta$ are the same in $G$ and $H$. We define a homomorphism $\func{\psi}{G(\Delta)}{H(\Delta)}$ by sending $\Delta_{(v)} \rightarrow e$ if $v \notin H$, and $\Delta_{(v)} \rightarrow \Delta_{(v)}$ otherwise. This is well defined for the same reasons as above, as the identity commutes with $\Delta_{(v)}$ for all $v$, and by definition the composite $\psi \circ \iota$ maps $\Delta_{(v)} \rightarrow \Delta_{(v)}$ for $v \in H$, so it is the identity. As it is injective, $\iota$ must be injective, so we are done.
\end{proof}
As an immediate corollary, $\Delta \leq G(\Delta)$, so that $\Delta$ must be residually finite if $G(\Delta) \rtimes \Gamma$ is.

\section{General Graph Wreath Products}
Under the same setup as in Section 2, we begin by constructing three necessary conditions for residual finiteness of graph wreath products.
\begin{theorem}
    If $\Delta$ is non-abelian and there is some $v \in V$ such that for every finite index subgroup $K \leq \Gamma$, there is some $w \in N(v)$ such that $w \in Kv$, then $G(\Delta) \rtimes \Gamma$ is not residually finite.
\end{theorem}
\begin{proof}
    Let $g, h \in \Delta$ not commute and consider $[g_{(v)}, h_{(v)}]$. This is not the identity as $\Delta$ is non-abelian, but for any finite index normal subgroup $M \leq G(\Delta)\rtimes \Gamma$, set $K = M \cap \Gamma$. If we work in $(G(\Delta)\rtimes \Gamma)/M$ then, setting $k \in K$ to be such that $w = kv$ $$[g_{(v)}, h_{(v)}] = [g_{(v)}, kh_{(v)}k^{-1}] = [g_{(v)}, h_{(kv)}] = e$$ so $[g_{(v)}, h_{(v)}]$ is in all finite index normal subgroups of $G(\Delta) \rtimes \Gamma$, the latter is not residually finite.
\end{proof}
Stating the contrapositive explicitly,
\begin{corollary}
    If $G(\Delta) \rtimes \Gamma$ is residually finite, then $\Delta$ is abelian or for all $v \in V$ there is a finite index subgroup $K \leq \Gamma$ such that $Kv \cap N(v) = \emptyset$.
\end{corollary}

\begin{theorem}
    Suppose there are $v, w \in V$ which are not neighbours and not equal, and for every finite index $K \leq \Gamma$ there is some $k \in K$ such that $kw$ neighbours or is $v$. Then $H = G(\Delta) \rtimes \Gamma$ is not residually finite.
\end{theorem}
\begin{proof}
    Take $g \in \Delta$ and consider $[g_{(v)}, g_{(w)}]$. This is non-trivial in $H$ as it is non-trivial in $G(\Delta)$ by Lemma 2.2. Suppose $M \leq H$ is a finite index normal subgroup, let $K = M \cap \Gamma$, and take $k \in K$ as in the hypothesis. Then in $H/M$ $$[g_{(v)}, g_{(w)}] = [g_{(v)}, kg_{(w)}k^{-1}] = [g_{(v)}, g_{(kw)}] = e$$ so $[g_{(v)}, g_{(w)}]$ is trivial in every finite quotient, so $H$ is not residually finite.
\end{proof} 
Stating the contrapositive explicitly,
\begin{corollary}
    If $G(\Delta) \rtimes \Gamma$ is residually finite, then for all $v, w \in V$ not neighbouring there is a finite index $K \leq \Gamma$ such that $Kw \cap (N(v) \cup \{v\}) = \emptyset$.
\end{corollary}
\begin{theorem}
    If there are $v, w \in V$ are distinct such that for all finite index normal $K \leq \Gamma$, $w \in Kv$, then $G(\Delta) \rtimes \Gamma$ is not residually finite.
\end{theorem}
\begin{proof}
    Given a finite index normal subgroup $M \leq G(\Delta) \rtimes \Gamma$, let $K = M \cap \Gamma$, and fix $k \in K$ such that $kw = v$. We consider $h = g_{(v)}g_{(w)}^{-1}$ for fixed $e \neq g \in \Delta$. In $(G(\Delta) \rtimes \Gamma)/M$, this is equal to $g_{(v)}kg_{(w)}^{-1}k^{-1} = g_{(v)}g_{(kw)}^{-1} = g_{(v)}g_{(v)}^{-1} = e$, so $h$ is trivial in every finite quotient, so we are done.
\end{proof}
\begin{corollary}
    If $G(\Delta) \rtimes \Gamma$ is residually finite, then for all $v, w \in V$ there is a finite index normal $K \leq \Gamma$ such that $w \notin Kv$.
\end{corollary}
\begin{remark}
    Corollaries 3.1.1 and 3.2.1 already imply this condition if $\Delta$ is non-abelian or $v, w$ are non-neighbouring.
\end{remark}
\begin{proof}
    \textit{(of Theorem 1.1)} The forwards direction is Corollaries 3.1.1, 3.2.1, 3.3.1.\\
    For the backwards direction, let $H = G(\Delta) \rtimes \Gamma$ and fix $(w, \gamma) \in H$ and enumerate the orbits of $G$ on which $w$ is supported as $G_1, \dots, G_n$. Define a graph $G'$ to be the induced subgraph of $G$ on $G_1 \cup \dots \cup G_n$. Then $\Gamma$ acts on $G'$ via its action on $G$, so we can define a homomorphism $\func{\phi}{H}{G'(\Delta) \rtimes \Gamma}$ by sending each vertex of some $G_i$ to its respective vertex, and sending every other orbit to the identity. This is a homomorphism as it respects all the commutativity relations in $G(\Delta)$, and respects all the relations introduced by the graph wreath products as $G'$ is preserved by the action of $\Gamma$, so $G'(\Delta)$ is preserved by the action of $\Gamma$ on $G(\Delta)$, and $\phi(w, \gamma)$ is non-trivial as either $\gamma$ is, and we're done, or $w$ is non-trivial, so it is supported on $G'$ by definition, and $\phi$ is injective on $G'(\Delta)$ by Lemma 2.2. So we may assume we only have finitely many orbits.
    \begin{lemma}
        Under the hypotheses and notation of Theorem 1.1, given finite sets $E \subseteq \Gamma$ and $A \subseteq G$, there is a finite index normal subgroup $K \leq \Gamma$ such that $\Gamma \rightarrow \Gamma/K$ is injective on $E$ and $G \rightarrow K\setminus G$ restricts to an isomorphism on the induced subgraph spanned by $A$.
    \end{lemma}
    \begin{proof}
        Define a finite index normal subgroup $K$ of $\Gamma$ as the intersection of the following finite list of subgroups:
        \begin{enumerate}
            \item If $\Delta$ is abelian, for each pair of neighbouring vertices in $\operatorname{supp} (w)$, take a finite index normal subgroup as guaranteed by 1.1(2);
            \item If $\Delta$ is non-abelian, for each orbit, choose a vertex $v$ in it and take a normal finite index $K'$ as guaranteed in 1.1(2) (For any other $w$ in the same orbit as $v$, $K'$ also satisfies 1.1(2) with respect to $w$ by normality of $K'$);
            \item If $v, v' \in \operatorname{supp} (w)$ are not neighbouring and not equal, include $K_{vv'}$ as guaranteed by 1.1(3);
            \item Some normal finite index subgroup such that the quotient is injective on $E$.
        \end{enumerate}
        This is finite index and normal. Form the quotient graph of $G$ by the action of $K$, ie the graph $K \setminus G$. We claim this graph quotient restricts to an isomorphism on $A$. If $a, b \in A$ are not adjacent, then we have that $\emptyset = Ka \cap (N(b) \cup \{b\})$ so that, as $ka, k'b$ are adjacent iff $k'^{-1}ka, b$ are, no elements of $Ka, Kb$ are adjacent. Furthermore, no two vertices in the support are sent to the same vertex in the image by the inclusion of subgroups of types (1)-(3). 
    \end{proof}
    Returning to the proof of Theorem 1.1, we obtain a $K$ from the above lemma by applying it in the case $E = \{e, \gamma\}$ and $A = \operatorname{supp}(w)$.
    \\Define a homomorphism $\phi \colon H \rightarrow (K \setminus G)(\Delta) \rtimes \Gamma/K$ as follows. We define the homomorphism on $G(\Delta)$ by $\Delta_{(v)} \mapsto \Delta_{(Kv)}$ and on $\Gamma$ by the quotient map. \\We claim $\phi$ is indeed a homomorphism. Note that the action of $\Gamma$ on $G$ descends by definition of the quotient graph to an action of $\Gamma/K$ on $K \setminus G$. This is also seen to respect the commutativity relations by definition of a quotient graph, with the exception of if $\Delta_x, \Delta_{x'}$ share an edge, and $Kx = Kx'$. If $\Delta$ is commutative then the commutativity relations are respected and the homomorphism is well defined. Otherwise we may write $x = kx'$ for some $k \in K$, but we know these can't neighbour by the inclusion of the subgroups of the second type corresponding to the orbit of $x$, so there is no commutativity relation needing to be respected.\\The subgroup $(K \setminus G)(\Delta)$ is finite index in $(K \setminus G)(\Delta) \rtimes \Gamma/K$, so $(K \setminus G)(\Delta) \rtimes \Gamma/K$ is residually finite. It remains to prove $\phi(w, \gamma)$ is non-trivial; if so the result follows by composing $\phi$ with some map to a finite group which maps $\phi(w, \gamma)$ to a non-identity element. If $\gamma \neq e$, then $\gamma$ is non-trivial in $\Gamma/K$ by the conclusion of Lemma 3.4. If $w \neq e$, then as we know the induced graph on $\operatorname{supp} (w)$ in $G$ is isomorphic to the induced graph on $K(\operatorname{supp} (w))$ in $K\setminus G$, by definition, $\phi$ restricts to an isomorphism on $\operatorname{supp} (w)$ by Lemma 2.2, so $w$ is non-trivial in $(K \setminus G)(\Delta)$, so we are done.
\end{proof}
There are two special cases within this which we shall state explicitly. By a regular graph wreath product we mean one in which $G$ has vertices corresponding to the elements of $\Gamma$ and $\Gamma$ acts by left multiplication.
\begin{corollary}
    For regular graph wreath products, $G(\Delta) \rtimes \Gamma$ is residually finite iff \begin{enumerate}
        \item $\Delta$ is abelian or there is some finite index subgroup such that $K \cap N(e) = \emptyset$.
        \item For all $w \in V$ not neighbouring the identity, there is a finite index $K \leq \Gamma$ such that $Kw \cap (N(e) \cup \{e\}) = \emptyset$.
    \end{enumerate}
\end{corollary}
\begin{proof}
    For all neighbouring $v, w \in V$, residual finiteness of $\Gamma$ guarantees a finite index subgroup such that $w \notin Kv$, and both other modifications come from the transitivity of $G$.
\end{proof}
\begin{corollary}
    \textit{(Theorem 1.3)} The permutational wreath product, $\Delta \wr_X \Gamma$ is residually finite iff both $\Delta$ and $\Gamma$ are, and either $\Delta$ is abelian and for all $x, y \in X$ there is a finite index subgroup $K \leq \Gamma$ such that $x \notin Ky$, or all vertex-stabilisers are finite index in $\Gamma$
\end{corollary}
\begin{proof}
    This follows from Theorem 1.1 as in the case of a wreath product (ie the case of a complete graph) condition 1.1(3) is vacuous, and 1.1(2) reads ``Either $\Delta$ is abelian and for all $v, w \in V$ there is a finite index subgroup $K_{vw} \leq \Gamma$ such that $w \notin Kv$, or each vertex stabiliser is trivial" which (intersecting over all $K_{vw}$ for fixed $v$) can be seen to be equivalent. 
\end{proof}
\begin{remark}
    This corollary was also proven (albeit by different methods) in \cite{C14}.
\end{remark}
\section{Connections to the case of LEF groups}
\begin{definition}
    A group $\Gamma$ is \textit{locally embeddable into finite groups} (abbreviated LEF) if for all finite subsets $E \subseteq \Gamma$ there is a finite group $Q_E$ and an injective function $\func{f}{E}{Q_E}$ such that if $g, h,gh \in E$ then $f(g)f(h) = f(gh)$ 
\end{definition}
We wish to connect our study of graph wreath products of residually finite groups to those of LEF groups (first defined in \cite{GV97}). The following definition is needed to state the criteria for LEF groups.
\begin{definition}
    An action $\Gamma \acts G$ is an \textit{LEF action} if for all finite subsets $A \subseteq \Gamma$ and $E \subseteq G$, there is a finite group $Q$, a function $\func{\phi}{A}{Q}$ such that if $a, b, ab \in A$ then $\phi(a)\phi(b) = \phi(ab)$, a finite graph $Y$ and an inclusion $\func{\psi}{E}{Y}$ of an induced subgraph such that for all $a \in A, e \in E$ if $a\cdot e \in E$, then $\phi(a)\cdot \psi(e) = \psi(a \cdot e)$
\end{definition}
The following theorem was proven in \cite{AB}.
\begin{theorem}
    Let $\Gamma$ and $\Delta$ be groups, and $G$ be a graph. Suppose $\Gamma\acts G$ via graph automorphisms. Then $G(\Delta) \rtimes \Gamma$ is LEF if and only if \begin{enumerate}
        \item $\Gamma, \Delta$ are LEF
        \item $\Gamma \acts G$ is an LEF action
    \end{enumerate}
\end{theorem}
It is easy to see that all residually finite groups are LEF (indeed, the injective function can even be taken to be a homomorphism from $G$ to $Q_E$ which is injective on $E$) so we would expect that the criteria of Theorem 1.1 should imply the criteria of Theorem 4.1. As explained above, 1.1(1) implies 4.1(1), so it remains to show that the remaining criteria of 1.1 imply 4.1(2)
\begin{proposition}
    Under the notation and conditions of 1.1, $\Gamma \acts G$ is an LEF action.
\end{proposition}
\begin{proof}
    
    Let $E \subseteq G$ and $A \subseteq \Gamma$ be finite. Without loss of generality, $G$ has finitely many $\Gamma$-orbits (indeed, we may quotient out by any orbits which do not contain an element of $E$). By Lemma 3.4, we obtain a finite index normal subgroup $K \leq \Gamma$ such that $G \rightarrow K\setminus G$ is injective on $A$ and $\Gamma \rightarrow \Gamma/K$ is injective on $E$. We claim that we may take $Q = \Gamma/K$, $Y = K \setminus G$ and $\phi, \psi$ to be the quotient maps restricted to $E$ and $A$ respectively. By definition, $\psi$ and $\phi$ are injective, and as $K$ is finite index, $Q$ and $Y$ are finite. By the same argument as in the proof of 1.1, $\Gamma/K$ acts on $K \setminus G$ by quotienting the action of $\Gamma$ on $G$ by $K$, and hence $\Gamma \acts G$ is an LEF action
\end{proof}
This links Theorem 1.1 with \cite{AB}, concluding the primary aim of this section.\\
Using Theorems 1.1 and 4.1 together, we may construct examples of non-finitely presented LEF groups by producing LEF graph wreath products that are not residually finite (and hence are not finitely presented) due to the following result proven in \cite{GV97}
\begin{proposition}
    Finitely presented LEF groups are residually finite.
\end{proposition}
We construct an example via this method.
\begin{example}
    Consider the graph wreath product of Example 1.2. From that example, we know it is not residually finite. We claim it is LEF. Let $E \subseteq G$ and $A \subseteq \Gamma$ be given finite subsets. List all $n \in \mathbb N$ such that for some $i \in E$, $i + n! \in E$ and call the set $B$. Fix $m \in \mathbb Z$ sufficiently large so that all elements of $A \cup (B + E)$ (where $B + E$ denotes all the sums of the form $b +e$ for $b \in B, e \in E$) are distinct modulo $m$. Define $H$ to be the subgraph of $G$ on the same vertices, but with only the edges coming from $n \in B $. We then get maps $E \rightarrow m\mathbb Z\setminus H$ and $A \rightarrow \Gamma/m\mathbb Z$ which are injective and present $E$ as an induced subgraph of $m\mathbb Z\setminus H$. Therefore $\Gamma \acts G$ is LEF, and hence is not finitely presented.
\end{example}
This is not an exhaustive method for constructing non-finitely presented graph wreath products as we will explain. From \cite{B61} for the regular case and \cite{C06} for the general case, we have the following: \begin{theorem}
    \textit{(Theorem 1.1 in \cite{C06})} If $W \neq 1$, then $W \wr_X G$ is finitely presented if and only if: \begin{enumerate}
        \item $W$ and $G$ are finitely presented;
        \item $G$ acts on $X$ with finitely generated stabilisers;
        \item The product action $G \acts X^2$ has finitely many orbits.
    \end{enumerate}
\end{theorem}Combining this with Theorem 1.3, we obtain wreath products (and hence, graph wreath products) that are residually finite but not finitely presented. An example is demonstrated below.
\begin{example}
    Let $W = C_3$, $G = \mathbb Z$, $X =\mathbb Z$, and let $G \acts X$ by translation. If we let $x, y \in X$, taking $m > x - y$ we see that $y \notin (m\mathbb Z) x$, so that $W \wr_X G$ is residually finite, however by \cite{B61}, this regular wreath product is not finitely presented as $G$ is infinite.
\end{example}
The reason that the construction of non-finitely presented graph wreath products in this manner is largely ineffective is that a more general and comprehensive statement of when graph wreath products are finitely presented (due to Kropholler and Martino, \cite{KM16}) is known: 
\begin{theorem}
    \textit{(Theorem 2.4 in \cite{KM16})} A graph wreath product, $G(\Delta) \rtimes \Gamma$ is finitely presented if and only if \begin{enumerate}
        \item $\Delta$ is finitely presented,
        \item $\Gamma$ is finitely presented,
        \item $G$ has finitely many orbits of vertices and edges,
        \item Every vertex stabiliser is finitely generated.
    \end{enumerate}
\end{theorem}
\printbibliography
\end{document}